\documentclass[a4paper,11pt]{amsart}
\usepackage{amsthm, amssymb, color}
\usepackage{amsfonts, amscd}
\usepackage{epsfig,multicol}
\usepackage{hyperref}

\numberwithin{equation}{section}

\theoremstyle{plain}
\newtheorem{theo}{Theorem}[section]
\newtheorem{coro}{Corollary}[section]
\newtheorem{prop}{Proposition}[section]
\newtheorem{lemm}{Lemma}[section]

\theoremstyle{definition}

\newtheorem{rema}{Remark}[section]

\def\Z{\mathbb Z}
\def\C{\mathbb C}
\def\R{\mathbb R}
\def\Q{\mathbb Q}

\def\MB{\mathcal B}

\newcommand{\uC}{\underline{\mathbb{C}}}

\DeclareMathOperator{\Aut}{Aut}

\def\ya{y^A}
\def\yb{y^B}
\def\xa{x^A}
\def\xb{x^B}
\def\xp{x^{PAP^{-1}}}
\def\yp{y^{PAP^{-1}}}
\def\q{q}
\def\ala{\alpha^A}
\def\alb{\alpha^B}
\def\height{{\rm ht}}

\begin{document}

\title[Invariance of Pontrjagin classes]{Invariance of Pontrjagin classes\\ for Bott manifolds}
\author[S. Choi]{Suyoung Choi}
\address{Department of Mathematics, Ajou University, Republic of Korea}
\email{schoi@ajou.ac.kr}
\author[M. Masuda]{Mikiya Masuda}

\address{Department of Mathematics, Osaka City University, Sumiyoshi-ku, Osaka 558-8585, Japan}
\email{masuda@sci.osaka-cu.ac.jp}
\author[S. Murai]{Satoshi Murai}
\address{Department of Mathematical Science, Faculty of Science, Yamaguchi University, 1677-1 Yoshida, Yamaguchi 753-8512, Japan}
\email{murai@yamaguchi-u.ac.jp}

\thanks{The first author partially supported by Basic Science Research Program through the National Research Foundation of Korea(NRF) funded by the Ministry of Education(NRF-2011-0024975) and the TJ Park Science Fellowship funded by the POSCO TJ Park Foundation.}
\thanks{The second author was partially supported by JSPS KAKENHI 25400095.}
\thanks{The third author was partially supported by JSPS KAKENHI 25400043.}

\maketitle

\begin{abstract}
A Bott manifold is the total space of some iterated $\C P^1$-bundle over a point.  We prove that any graded ring isomorphism between the cohomology rings of two Bott manifolds preserves their Pontrjagin classes.  Moreover, we prove that such an isomorphism is induced from a diffeomorphism if the Bott manifolds are $\Z/2$-trivial, where a Bott manifold is called $\Z/2$-trivial if its cohomology ring with $\Z/2$-coefficient is isomorphic to that of a product of $\C P^1$'s.
\end{abstract}

\section{Introduction}

One of the fundamental problems in topology is to classify manifolds (up to diffeomorphism, homeomorphism etc.) by invariants and cohomology rings are not sufficient to classify them in general.  For example, the surgery theory tells us  that there are infinitely many diffeomorphism types in the family of closed smooth manifolds homotopy equivalent to a complex projective space $\C P^n$ when $n\ge 3$.  However, the surgery theory further tells us that they are distinguished by their Pontrjagin classes up to finite ambiguity, and this is true in general for the family of closed smooth manifolds homotopy equivalent to a fixed closed smooth manifold $X$ where $X$ is simply connected and of dimension $\ge 5$.

On the other hand, we have a feeling that most of closed smooth manifolds do not admit an effective smooth $S^1$-action.  For example, T. Petrie \cite{petr72} conjectures that if $M$ is a closed smooth manifold homotopy equivalent to $\C P^n$ and $M$ admits an effective smooth $S^1$-action, then a homotopy equivalence $f\colon M\to \C P^n$ preserves their Pontrjagin classes.   Note that $\C P^n$ has an effective smooth action of $(S^1)^n$.  The conjecture is not solved but many partial affirmative solutions are obtained.  Among them, Petrie \cite{petr73} shows that the conjecture is true if $M$ admits an effective smooth action of $(S^1)^n$, see \cite{mani_atlas} for more details.

A complete non-singular toric variety (simply called a toric manifold) of complex dimension $n$ admits an effective algebraic action of $(\C^*)^n$, in particular, an effective smooth action of $(S^1)^n$.  The complex projective space $\C P^n$ is a typical example of a toric manifold.  Motivated by the Petrie's conjecture and his result mentioned above, the second named author and D.~Y.~Suh \cite{ma-su08} posed a problem which asks whether any cohomology ring isomorphism between toric manifolds preserves their Pontrjagin classes.  Little is known about the problem.  One of our main purposes is to show that the problem is affirmative for some nice class of toric manifolds called Bott manifolds.  We even show that any cohomology ring isomorphism is induced from a diffeomorphism for a certain subclass of Bott manifolds.

A Bott tower of height $n$ is a sequence of $\C P^1$-bundles
\begin{equation} \label{eq:1.1}
   B_n \stackrel{\pi_n}\longrightarrow B_{n-1} \stackrel{\pi_{n-1}}\longrightarrow \cdots \stackrel{\pi_2}\longrightarrow B_1 \stackrel{\pi_1}\longrightarrow B_0 = \{\text{a point}\},
\end{equation}
where $B_j$ is the projectivization $P(\uC\oplus L_j)$ of a trivial complex line bundle $\uC$ and a complex line bundle $L_j$ over $B_{j-1}$, and $\pi_j \colon B_j \to B_{j-1}$ is the projection for $j=1,2,\dots,n$.  It is known that $B_n$ is a toric manifold and $B_n$ is called an \emph{$n$-stage Bott manifold} or simply a \emph{Bott manifold}.

If all the fibrations in \eqref{eq:1.1} are trivial, then $B_n$ is diffeomorphic to $(\C P^1)^n$. The $1$-stage Bott manifold is $\C P^1$ and $2$-stage Bott manifolds are Hirzebruch surfaces.  As is well-known, there are only two diffeomorphism types among Hirzebruch surfaces; they are $(\C P^1)^2$ and $\C P^2\#\overline{\C P^2}$, where $\overline{\C P^2}$ is $\C P^2$ with the opposite orientation.  However, there are infinitely many diffeomorphism types among $n$-stage Bott manifolds when $n\ge 3$, and it is an interesting open question to classify them up to diffeomorphism or homeomorphism (\cite{cr-kr11}).

Our first main result is the following.

\begin{theo} \label{theo:0.1}
Any graded ring isomorphism between the cohomology rings (with integer coefficients) of two Bott manifolds preserves their Pontrjagin classes.
\end{theo}

If a graded ring isomorphism between the cohomology rings of two smooth manifolds is induced from a diffeomorphism, then the isomorphism preserves their Pontrjagin classes.  So, Theorem~\ref{theo:0.1} provides a supporting evidence to the following conjecture.

\medskip
\noindent
{\bf Strong cohomological rigidity conjecture for Bott manifolds.}  Any graded ring isomorphism between the cohomology rings (with integer coefficients) of two Bott manifolds is induced from a diffeomorphism.

\medskip

The conjecture above in particular claims that two Bott manifolds are diffeomorphic if their cohomology rings (with integer coefficients) are isomorphic as graded rings, and we call this weaker conjecture \emph{Cohomological rigidity conjecture for Bott manifolds}. No counterexamples are known to these conjectures and some partial affirmative solutions are known; for instance, the strong cohomological rigidity conjecture is affirmative for $\Q$-trivial Bott manifolds, where a Bott manifold $B_n$ is called \emph{$\Q$-trivial} if $H^*(B_n;\Q)\cong H^*((\C P^1)^n;\Q)$ as graded rings (\cite{ch-ma12}).  It is also affirmative up to $3$-stage Bott manifolds and the cohomological rigidity conjecture is affirmative for $4$-stage Bott manifolds (\cite{choi11}).

We say that a Bott manifold $B_n$ is \emph{$\Z/2$-trivial} if $H^*(B_n;\Z/2)\cong H^*((\C P^1)^n;\Z/2)$ as graded rings, where $\Z/2=\Z/2\Z$. Our second main result is the following.

\begin{theo} \label{theo:0.2}
The strong cohomological rigidity conjecture is affirmative for $\Z/2$-trivial Bott manifolds, namely any graded ring isomorphism between the cohomology rings (with integer coefficients) of two $\Z/2$-trivial Bott manifolds is induced from a diffeomorphism.
\end{theo}

There are infinitely many diffeomorphism types among $\Z/2$-trivial $n$-stage Bott manifolds when $n\ge 3$ while there are only finitely many diffeomorphism types among $\Q$-trivial $n$-stage Bott manifolds for any $n$.  Therefore, the family of $\Z/2$-trivial Bott manifolds is much larger than that of $\Q$-trivial Bott manifolds but the former family does not contain the latter, for instance, a 2-stage Bott manifold $\C P^2\#\overline{\C P^2}$ is not $\Z/2$-trivial but $\Q$-trivial.

The rigidity conjectures mentioned above are posed (as problems) more generally for toric manifolds or some related family of manifolds.  The strong cohomological rigidity does not hold for arbitrary toric manifolds while no counterexamples are known to the cohomological rigidity problems for toric manifolds. A real analogue of the rigidity problems is also studied.  See survey papers \cite{ch-ma-su12} and \cite{ma-su08} for details.

This paper is organized as follows.  In Section~\ref{sect:2} we review some known facts on the cohomology rings of Bott manifolds.  In Section~\ref{sect:3} we  introduce new bases of the cohomology rings and restate some facts mentioned in Section~\ref{sect:2}.  In Section~\ref{sect:4} we analyze graded ring isomorphisms between the cohomology rings of two Bott manifolds, where the new bases introduced in Section~\ref{sect:3} play a role.  We prove Theorem~\ref{theo:0.1} in Section~\ref{sect:5} and Theorem~\ref{theo:0.2} in Section~\ref{sect:6}.  Finally, in Section~\ref{sect:7} we make some remarks on automorphisms of the cohomology ring of a Bott manifold, which clarifies our difficulty to solve the strong cohomological rigidity conjecture for Bott manifolds completely.

\section{Cohomology rings of Bott manifolds} \label{sect:2}

In this section we will recall some known facts on the cohomology rings of Bott manifolds and the quotient construction of Bott manifolds.

We denote by $\alpha_j$ the first Chern class of the complex line bundle $L_j$ used to construct the Bott tower \eqref{eq:1.1}.
It follows from the Borel-Hirzebruch formula (\cite{bo-hi58}) that $H^\ast(B_j;\Z)$ is a free module over $H^\ast(B_{j-1};\Z)$ through the map $\pi_j^\ast\colon H^*(B_{j-1};\Z)\to H^*(B_j;\Z)$ on two generators $1$ and $x_j$ of degree $0$ and $2$ respectively, where $x_j$ is the first Chern class of the tautological line bundle $\gamma_j$ over $B_j$, and that the ring structure is determined by the single relation
$$
    x_j^2 = \pi_j^\ast(\alpha_j) x_j.
$$
Using the formula inductively on $j$ and regarding $H^*(B_j;\Z)$ as a subring of $H^*(B_n;\Z)$ through the projections in \eqref{eq:1.1}, we see that
\begin{equation} \label{eq:1.2}
    H^*(B_n;\Z)=\Z[x_1, \dots, x_n]/(x_j^2 - \alpha_jx_j \mid j=1,2,\dots,n)
\end{equation}
where $\alpha_1=0$.
The following lemma easily follows from \eqref{eq:1.2}.

\begin{lemm} \label{lemm:1.1}
Let $k$ be a positive integer less than or equal to $n$.  Then the set
\[
\{x_{i_1}x_{i_2}\dots x_{i_k}\mid 1\le i_1<i_2<\dots <i_k\le n\}
\]
is an additive basis of $H^{2k}(B_n;\Z)$.
\end{lemm}

The Pontrjagin class of a Bott manifold has a simple expression as is shown in the following lemma.

\begin{lemm} \label{lemm:pBn}
The Pontrjagin class $p(B_n)$ of the Bott manifold $B_n$ is given by
\begin{equation*} \label{eq:1.3}
p(B_n)=\prod_{j=1}^n(1+(2x_j-\alpha_j)^2)=\prod_{j=1}^n(1+\alpha_j^2) \in H^*(B_n;\Z)
\end{equation*}
where $(2x_j-\alpha_j)^2=\alpha_j^2$ because $x_j^2=\alpha_jx_j$.
\end{lemm}

\begin{proof}
The lemma is known but since there seems no literature which mentions the formula explicitly, we shall give a proof.

Since $\pi_j\colon B_j\to B_{j-1}$ is the projectivization of the Whitney sum of the trivial line bundle $\uC$ and the line bundle $L_j$ over $B_{j-1}$,  the tangent bundle $TB_j$ of $B_j$ splits into
\begin{equation} \label{eq:pBj}
TB_j=T_f B_j\oplus \pi_j^*(TB_{j-1})
\end{equation}
where $T_fB_j$ denotes the complex line bundle along the fibers of the fiber bundle $\pi_j\colon B_j\to B_{j-1}$ and $\pi_j^*(TB_{j-1})$ is the pullback of the tangent bundle $TB_{j-1}$ of the base space $B_{j-1}$ by the projection $\pi_j$.  Since $x_j$ is the first Chern class of the tautological line bundle $\gamma_j$ over $B_j$ and the total Chern class of $\uC\oplus L_j$ is $1+\alpha_j$, it follows from \cite[(2) in p.515]{bo-hi58} that the total Chern class of the complex line bundle $T_fB_j$ is given by $1-2x_j+\alpha_j$ and hence its total Pontrjagin class is given by $1+(2x_j-\alpha_j)^2$. This together with \eqref{eq:pBj} shows that
\[
p(B_j)=(1+(2x_j-\alpha_j)^2)p(B_{j-1})
\]
because $H^\ast(B_n;\Z)$ has no 2-torsion. Then the lemma follows by applying the above formula inductively on $j$.
\end{proof}

For $u=\sum_{i=1}^nc_ix_i \in H^2(B_n;\Z)$ with $c_i\in \Z$, we define
\[
\height(u)=\max\{i\in [n]=\{1,2,\dots,n\}\mid c_i\not=0\}
\]
and call it the \emph{height} of $u$.  Note that $\height(\alpha_j)<\height(x_j)=j$.
We say that a pair of primitive elements in $H^2(B_n;\Z)$ is a \emph{primitive vanishing pair} if the product of the elements vanishes.  Note that a pair $(x_j, x_j-\alpha_j)$ is a primitive vanishing pair for any $j\in [n]$.

\begin{lemm}[Lemma 2.2 in \cite{ch-ma12}] \label{lemm:1.3}
A primitive vanishing pair is of the form
\[
\text{$(ax_j+u,a(x_j-\alpha_j)-u)$ or $(ax_j+u,-a(x_j-\alpha_j)+u)$}
\]
where $a\in\Z\backslash\{0\}$, $u\in H^2(B_n;\Z)$, $u(u+a\alpha_j)=0$ and $\height(u)<j$.
\end{lemm}

\begin{coro}[Corollary 2.1 in \cite{ch-ma12}] \label{coro:1.1}
Primitive square zero elements in $H^2(B_n;\Z)$ are of the form $\pm(x_j-\frac{1}{2}\alpha_j)$ if $\alpha_j\equiv 0\pmod{2}$ and $\pm(2x_j-\alpha_j)$ otherwise.
\end{coro}

We shall review the quotient construction of Bott manifolds (\cite[Prop.\ 3.1]{ci-ra05}).  Remember that $\alpha_1=0$ and express
\begin{equation} \label{eq:1.2-1}
\alpha_j=\sum_{i=1}^{j-1} A^i_j x_i \quad  \text{for $j=2, 3,\dots, n$}
\end{equation}
with $A^i_j \in \Z$.
Let $S^{2d-1}$ be the unit sphere of $\C^d$ for $d=1,2$.
Then the Bott manifold $B_n$ in \eqref{eq:1.1} can be obtained as the quotient of $(S^3)^n$ by the free action of $(S^1)^n$ defined by
\begin{equation} \label{eq:1-a}
\begin{split}
&(g_1,g_2,\dots,g_n)\cdot \big((z_1,w_1),(z_2,w_2),\dots, (z_n,w_n)\big)\\
=& \big((g_1z_1,g_1w_1), ((g_1^{-A^1_2})g_2z_2,g_2w_2),\dots, ((\prod_{k=1}^{n-1}g_k^{-A^k_n})g_nz_n,g_nw_n)\big) 
\end{split}
\end{equation}
where $g_i\in S^1$ and $(z_i,w_i)\in S^3$ for $i=1,2,\dots,n$.  Projections
\[
\big((z_1,w_1),\dots, (z_n,w_n)\big)\to \big((z_1,w_1),\dots, (z_{n-1},w_{n-1})\big)\to \dots \to (z_1,w_1)
\]
induce the Bott tower \eqref{eq:1.1}.  The tautological line bundle $\gamma_j$ over $B_j$ can be described as the quotient of the trivial complex line bundle $(S^3)^n\times \C\to (S^3)^n$ by the action of $(S^1)^n$ on the total space $(S^3)^n\times \C$ defined as follows:
\[
 g\cdot ((z,w),u)=(g\cdot(z,w),g_j^{-1}u)\qquad\text{$(u\in\C)$},
\]
where $g=(g_1,g_2,\dots,g_n)$, $(z,w)=\big((z_1,w_1),(z_2,w_2),\dots, (z_n,w_n)\big)$ and $g\cdot (z,w)$ denotes the action defined in \eqref{eq:1-a}.

\section{Base change} \label{sect:3}

We set
\begin{equation} \label{eq:2.1}
y_j=x_j-\frac{1}{2}\alpha_j \quad\text{for $j=1,2,\dots,n$.}
\end{equation}
Then, we have
\begin{equation} \label{eq:2.1-0}
4y_j^2=(2x_j-\alpha_j)^2=\alpha_j^2 \quad\text{ in $H^\ast(B_n;\Z)$},
\end{equation}
and it follows from Lemma~\ref{lemm:pBn} that
\begin{equation} \label{eq:2.1-1}
p(B_n)=\prod_{i=1}^n(1+4y_j^2).
\end{equation}

If $\alpha_j\equiv 0\pmod{2}$, then $y_j$ is integral, that is, an element of $H^2(B_n;\Z)$.

\begin{lemm} \label{lemm:2.1}
If $\alpha_i\equiv \alpha_j\equiv 0\pmod{2}$ and $y_i\equiv y_j\pmod{2}$, then $i=j$.
\end{lemm}

\begin{proof}
Since $\height(\alpha_i)<i$ and $\height(\alpha_j)<j$, the assumption $y_i\equiv y_j\pmod{2}$ implies $x_i\equiv x_j\pmod{2}$ and hence $i=j$ because $\{x_1,x_2,\dots,x_n\}$ is an additive basis of $H^2(B_n;\Z)$ (cf. Lemma~\ref{lemm:1.1}).
\end{proof}

Let $A^i_j$ $(i<j)$ be the integers defined in \eqref{eq:1.2-1}.  Setting $A^i_j=0$ for $i\ge j$, we obtain an  integral strictly upper triangular $n\times n$ matrix $A$ with $A^i_j$ as the $(i,j)$ entry.
Then it follows from \eqref{eq:2.1} that
\[
(y_1,\dots,y_n)=(x_1,\dots,x_n)(E-\frac{1}{2}A),
\]
where $E$ is the identity matrix.
Here $E-\frac{1}{2}A$ is an upper triangular unipotent matrix and so is its inverse.  Therefore, if we denote the $(i,j)$ entry of $(E-\frac{1}{2}A)^{-1}$ by $a^i_j$, then
\begin{equation} \label{eq:2.2}
x_j=\sum_{i\le j}a^i_jy_i\quad\text{and}\quad
\frac{1}{2}\alpha_j=x_j-y_j=\sum_{i<j}a^i_jy_i.
\end{equation}
By \eqref{eq:2.1-0} and \eqref{eq:2.2},
we have
\begin{equation} \label{eq:2.3}
y_j^2 =\big( \frac{\alpha_j}{2} \big)^2 =  \big(\sum_{i<j}a^i_jy_i\big)^2\quad\text{in $H^*(B_n;\Q)$}.
\end{equation}

The following lemma easily follows from \eqref{eq:2.3} or Lemma~\ref{lemm:1.1}.

\begin{lemm} \label{lemm:2.2}
Let $k$ be a positive integer less than or equal to $n$.  Then the set
\[
\{y_{i_1}y_{i_2}\dots y_{i_k}\mid 1\le i_1<i_2<\dots <i_k\le n\}
\]
is an additive basis of $H^{2k}(B_n;\Q)$ over $\Q$.
\end{lemm}

\section{Cohomology ring isomorphisms} \label{sect:4}

Let $\MB_n$ be the set of integral strictly upper triangular $n\times n$ matrices.  Since the Bott manifold $B_n$ in \eqref{eq:1.1} is determined by a matrix $A\in \MB_n$, we will denote $B_n$ by $M(A)$.
For the zero matrix $O\in \MB_n$, $M(O)=(\C P^1)^n$ and it is known that if $M(A)$ is diffeomorphic to $M(O)$, then $A=O$ (\cite{ma-pa08}).  However, it happens that $M(A)$ and $M(B)$ are diffeomorphic even if $A,B\in \MB_n$ are different.

Henceforth the cohomology elements $x_j, y_j, \alpha_j$ for $M(A)$ will be denoted by $\xa_j, \ya_j,\ala_j$ respectively to avoid confusion.

\begin{prop} \label{prop:3.1}
Suppose that $\psi\colon H^*(M(A);\Q)\to H^*(M(B);\Q)$ is a graded ring isomorphism. Then there are non-zero $\q_1,\dots,\q_n\in\Q$ and a permutation $\sigma$ on $[n]$ such that
\[
\psi(\ya_j)=\q_j\yb_{\sigma(j)}\quad \text{for $j=1,\dots,n$.}
\]
\end{prop}

\begin{proof}
We prove the proposition by induction on $j$.  When $j=1$, $\psi(\ya_1)^2=0$ because $(\ya_1)^2=0$.  Therefore $\psi(\ya_1)$ is a nonzero scalar multiple of $\yb_{\sigma(1)}$ for some $\sigma(1)\in [n]$ by Corollary~\ref{coro:1.1}.

Suppose that
\begin{equation} \label{eq:3.1}
\psi(\ya_j)=\q_j\yb_{\sigma(j)}\quad \text{for $j<k$}.
\end{equation}
Then, since
\[
(\q_j\yb_{\sigma(j)})^2=\psi(\ya_j)^2=\psi\big((\ya_j)^2\big)
=\psi\big((\sum_{i<j}a^i_j\ya_i)^2\big)=\big(\sum_{i<j}a^i_j\q_i\yb_{\sigma(i)}\big)^2,
\]
one can inductively see that for $j<k$,
\begin{equation} \label{eq:3.2}
(\yb_{\sigma(j)})^2 \text{ is a linear combination of $\yb_{\sigma(i_1)}\yb_{\sigma(i_2)}$ with $i_1<i_2<j$.}
\end{equation}

Set $S=\{\sigma(1),\dots,\sigma(k-1)\}$.
Since $\psi$ is an isomorphism, it follows from \eqref{eq:3.1} that $\psi(\ya_k)$ is not a linear combination of $\yb_{\sigma(1)},\dots,\yb_{\sigma(k-1)}$.  Therefore, there is $m\in[n]\backslash S$ such that
\begin{equation} \label{eq:3.3}
\psi(\ya_k)=\sum_{s\le m\text{ or } s\in S}b_s\yb_s \quad(b_s\in \Q \text{ and }b_m\not=0).
\end{equation}

We consider squares at the both sides above.  It follows from \eqref{eq:2.3} and \eqref{eq:3.1} that
\[
\psi(\ya_k)^2=\psi\big((\ya_k)^2\big)=\psi\big((\sum_{i<k}a^i_k\ya_i)^2\big)=\big(\sum_{i<k}a^i_k\q_i\yb_{\sigma(i)}\big)^2
\]
and hence
\begin{equation} \label{eq:3.4}
\psi(\ya_k)^2 \text{ is a linear combination of  $\yb_{\sigma(i_1)}\yb_{\sigma(i_2)}$ with $i_1<i_2<k$}
\end{equation}
by \eqref{eq:3.2}. On the other hand, we claim that if $b_t\not=0$ for some $t<m$ or $t\in S$ at the right hand side of \eqref{eq:3.3}, then a non-zero scalar multiple of $\yb_m\yb_t$ appears in  the square of the right hand side of \eqref{eq:3.3}. Indeed, $(\yb_s)^2$ for $s\le m$ is a linear combination of $\yb_i\yb_j$ with $i<j<m$ by \eqref{eq:2.3} and $(\yb_s)^2$ for $s\in S$ is a linear combination of $\yb_i\yb_j$ with $i\not=j\in S$ by \eqref{eq:3.2}.  Therefore the claim holds because $\{\yb_i\yb_j\mid 1\le i<j\le n\}$ is an additive basis of $H^4(M(B);\Q)$ by Lemma~\ref{lemm:2.2}.  However the claim contradicts \eqref{eq:3.4} since $m\notin S$.  Therefore, $\psi(\ya_k)=b_m\yb_m$.  This completes the induction step and proves the proposition.
\end{proof}

\begin{lemm} \label{lemm:3.1}
Suppose that the graded ring isomorphism $\psi$ in Proposition~\ref{prop:3.1} is induced from a graded ring isomorphism from $H^*(M(A);\Z)$ to $H^*(M(B);\Z)$.  Then the rational number $\q_j$ in Proposition~\ref{prop:3.1} belongs to $\{\pm\frac{1}{2}, \pm 1,\pm 2\}$ for $j=1,2,\dots,n$, and
\begin{enumerate}
\item $\q_j\in \{\pm\frac{1}{2}\}$ if and only if $\ala_j\not\equiv 0 \pmod{2}$ and $\alb_{\sigma(j)}\equiv 0\pmod{2}$,
\item $\q_j\in \{\pm 2\}$ if and only if $\ala_j\equiv 0\pmod{2}$ and $\alb_{\sigma(j)}\not\equiv 0\pmod{2}$.
\end{enumerate}
Moreover, if $\q_j=\pm 1$ for all $j$, then $\psi(p(M(A)))=p(M(B))$.
\end{lemm}

\begin{proof}
Since $\psi(\xa_j-\frac{1}{2}\ala_j)=\q_j(\xb_{\sigma(j)}-\frac{1}{2}\alb_{\sigma(j)})$ by Proposition~\ref{prop:3.1}, the former statement in the lemma follows from the fact that $\psi$ restricted to the cohomology rings with integer coefficients
sends primitive elements to primitive elements and the latter follows from \eqref{eq:2.1-1} and Proposition~\ref{prop:3.1}.
\end{proof}

We fix a graded ring isomorphism $\psi\colon H^*(M(A);\Z)\to H^*(M(B);\Z)$ by the end of Lemma~\ref{lemm:3.2}. Since $(\psi(\xa_j),\psi(\xa_j-\ala_j))$ is a primitive vanishing pair, it follows from Lemma~\ref{lemm:1.3} that there are $a\in\Z\backslash \{0\}$, $u\in H^2(M(B);\Z)$ and $k\in [n]$ such that
\begin{equation} \label{eq:3.5}
(\psi(\xa_j),\psi(\xa_j-\ala_j))=\begin{cases}&(a\xb_k+u,a(\xb_k-\alb_k)-u)\quad \text{ or }\\
& (a\xb_k+u,-a(\xb_k-\alb_k)+u), \end{cases}
\end{equation}
where $u(u+a\alb_k)=0$ and $\height(u)<k$. Remember that $a^i_j$ is the $(i,j)$ entry of the matrix $(E-\frac{1}{2}A)^{-1}$, in other words, $a^i_j$ is the rational number determined by \eqref{eq:2.2}.

With this understood, we have

\begin{lemm} \label{lemm:3.1-1}
In the former case of \eqref{eq:3.5}, $\q_j=a$ (hence $\q_j$ is an integer), $k=\sigma(j)$ and $a^i_j=0$ for $\sigma(i)>\sigma(j)$. In the latter case of \eqref{eq:3.5}, $\ala_j=2a\q_i^{-1}\ya_i$ for some $i<j$.
\end{lemm}

\begin{proof}
Since $2\ya_j=\xa_j+(\xa_j-\ala_j)$ and $\ala_j=\xa_j-(\xa_j-\ala_j)$, a simple computation using \eqref{eq:3.5} shows that
\begin{equation} \label{eq:3.5-0}
(\psi(2\ya_j), \psi(\ala_j))=\begin{cases} (2a\yb_k,2u+a\alb_k)\quad &\text{in the former case of \eqref{eq:3.5},}\\
(2u+a\alb_k,2a\yb_k)\quad &\text{in the latter case of \eqref{eq:3.5}.}\end{cases}
\end{equation}

In the former case of \eqref{eq:3.5}, we have $\psi(\ya_j)=a\yb_k$ by \eqref{eq:3.5-0}.  This together with Proposition~\ref{prop:3.1} shows that $\q_j=a$ and $k=\sigma(j)$.  Moreover, since $\psi(\ala_j)=2u+a\alb_k$ by \eqref{eq:3.5-0} and $\height(u)<k$,  we have
\begin{equation*} \label{eq:3.5-1}
\height(\psi(\ala_j))<k=\sigma(j)
\end{equation*}
while since $\ala_j=2\sum_{i<j}a^i_j\ya_i$ by \eqref{eq:2.2}, we have
$$\psi(\ala_j)=2\sum_{i<j}a^i_j\q_i\yb_{\sigma(i)}$$
by Proposition~\ref{prop:3.1}.  These show that $a^i_j=0$ for $\sigma(i)>\sigma(j)$.

In the latter case of \eqref{eq:3.5}, $\psi(\ala_j)=2a\yb_k$ by \eqref{eq:3.5-0}.  Therefore $\ala_j=2a\q_i^{-1}\ya_i$ for some $i\in [n]$ by Proposition~\ref{prop:3.1} and $j>\height(\ala_j)=\height(\ya_i)=i$.
\end{proof}

Note that $2a\q_i^{-1}$ in Lemma~\ref{lemm:3.1-1} is a nonzero integer because $a\in \Z\backslash\{0\}$ and $\q_i\in \{\pm \frac{1}{2}, \pm 1, \pm 2\}$ by Lemma~\ref{lemm:3.1}.  We say that $\alpha_j$ is of \emph{exceptional type} if $\alpha_j=c y_i$ for some nonzero integer $c$ and $i<j$ and is of \emph{even exceptional type} if the nonzero integer $c$ is even.

\begin{lemm} \label{lemm:3.2}
Suppose that $\q_j=\pm \frac{1}{2}$.  Then there are $i<j$ and $c\in\Z\backslash\{0\}$ such that $\ala_j=c\ya_i$. If $\ala_j$ is not of even exceptional type, then $c$ is odd and $\q_i=\pm 2$. If neither $\ala_j$ nor $\alb_{\sigma(i)}$ are of even exceptional type, then $\alb_{\sigma(i)}=d\yb_{\sigma(j)}$ with some odd integer $d$.
\end{lemm}

\begin{proof}
Since $\q_j=\pm\frac{1}{2}$, that is not an integer, by assumption, Lemma~\ref{lemm:3.1-1} says that the latter case of \eqref{eq:3.5} must occur and $\ala_j=2a\q_i^{-1}\ya_i$.  Setting $c=2a\q_i^{-1}$, we have
\begin{equation} \label{eq:3.8}
\ala_j=c\ya_i,
\end{equation}
where $c$ is a nonzero integer.  If $\ala_j$ is not of even exceptional type, then $c=2a\q_i^{-1}$ is odd and hence $\q_i$ must be $\pm 2$ (and $a$ is odd).

It remains to prove the last assertion in the lemma.  Suppose that neither $\ala_j$ nor $\alb_{\sigma(i)}$ are of even exceptional type.  Then $\q_i=\pm 2$ as observed above, and since $\psi(\ya_i)=\q_i\yb_{\sigma(i)}$ by Proposition~\ref{prop:3.1}, we have $\psi^{-1}(\yb_{\sigma(i)})=\pm\frac{1}{2}\ya_i$. Then the same argument as above applied for $\psi^{-1}(\yb_{\sigma(i)})=\pm\frac{1}{2}\ya_i$ tells us that there are $\ell\in [n]$ and $d\in\Z\backslash\{0\}$ such that
\begin{equation} \label{eq:3.9}
\alb_{\sigma(i)}=d\yb_{\ell}.
\end{equation}
Here $d$ is odd because $\alb_{\sigma(i)}$ is not of even exceptional type by assumption.
In the sequel, it suffices to prove $\ell=\sigma(j)$.

Since $d$ is odd, it follows from \eqref{eq:3.9} that $\alb_{\ell}\equiv 0 \pmod{2}$ while $\alb_{\sigma(j)}\equiv 0 \pmod 2$ by Lemma~\ref{lemm:3.1} since $\q_j=\pm \frac{1}{2}$. Therefore, it suffices to show $\yb_{\ell}\equiv \yb_{\sigma(j)}\pmod 2$ by Lemma~\ref{lemm:2.1}, which we shall check below.
Since $\pm\yb_{\sigma(j)}=\psi(2\ya_j)=\psi(2\xa_j-\ala_j)=2\psi(\xa_j)-\psi(\ala_j)$, we have
\begin{equation} \label{eq:3.10}
\yb_{\sigma(j)}\equiv \psi(\ala_j)\pmod 2.
\end{equation}
Similarly, since $\pm \psi(\ya_i)=2\yb_{\sigma(i)}=2\xb_{\sigma(i)}-\alb_{\sigma(i)}$, we have
\begin{equation} \label{eq:3.11}
\psi(\ya_i)\equiv \alb_{\sigma(i)}\pmod 2.
\end{equation}
Thus, since $c$ and $d$ are both odd integers, it follows from \eqref{eq:3.10}, \eqref{eq:3.8}, \eqref{eq:3.11} and \eqref{eq:3.9} that
\[
\yb_{\sigma(j)}\equiv \psi(\ala_j)\equiv \psi(\ya_i)\equiv\alb_{\sigma(i)}\equiv \yb_{\ell}\pmod 2,
\]
proving the desired congruence relation.
\end{proof}

\begin{lemm} \label{lemm:4.1}
For $A\in\MB_n$, there is $B\in \MB_n$ such that none of $\alb_1,\dots,\alb_n$ are of even exceptional type and there is a graded ring isomorphism $\psi\colon H^*(M(A);\Z)\to H^*(M(B);\Z)$ such that $\psi(p(M(A)))=p(M(B))$. 
\end{lemm}

\begin{proof}
Suppose that $\ala_j$ is of even exceptional type but $\ala_k$ for $k<j$ is not.  In the following we will find $B\in\MB_n$ such that $\alb_k$ for $k< j$ is not of even exceptional type but $\height(\alb_j)<\height(\ala_j)$ and that $H^*(M(A);\Z)$ and $H^*(M(B);\Z)$ are isomorphic as graded rings.  If $\alb_j$ is still of even exceptional type, then we repeat the argument until we reach $B$ such that $\alb_k$ for $k\le j$ is not of even exceptional type. Note that this can be achieved because if $\height(\alb_j)=0$, then $\alb_j=0$ which is not of even exceptional type.
Doing this procedure inductively on $j$, we finally reach the desired $B$ in the lemma.

Since $\ala_j$ is of even exceptional type by assumption, we have
\begin{equation} \label{eq:4.0-0}
\ala_j=c(\xa_i-\frac{1}{2}\ala_i)
\end{equation}
with some nonzero even integer $c$ and $i<j$.  We define a matrix $B$ of size $n$ by
\begin{equation*}
B_k^\ell=\begin{cases} A_k^\ell \quad &\text{if $k\not=j$ and $\ell\not=i$},\\
A^i_k+\frac{c}{2}A^j_k \quad &\text{if $k\not=j$ and $\ell=i$},\\
-\frac{c}{2}A_i^\ell\quad&\text{if $k=j$}.
\end{cases}
\end{equation*}
Since $A\in \MB_n$ and $i<j$, the matrix $B$ is indeed in $\MB_n$ (i.e., $B_k^\ell=0$ for $\ell\ge k$), and
\begin{equation} \label{eq:4.0-1}
\alb_k=\begin{cases} \sum_{\ell=1}^nA_k^\ell \xb_\ell+\frac{c}{2}A^j_k\xb_i \quad&\text{if $k\not=j$},\\
-\frac{c}{2}\sum_{\ell=1}^n A_i^\ell \xb_\ell \quad&\text{if $k=j$.}
\end{cases}
\end{equation}
Note that since $i<j$, $A^j_i=0$ and hence it follows from \eqref{eq:4.0-1} that
\begin{equation} \label{eq:4.0-1-1}
\alb_j=-\frac{c}{2}\alb_i.
\end{equation}

We define
\begin{equation} \label{eq:4.0-2}
\psi(\xa_\ell)=\begin{cases} \xb_\ell\quad&\text{if $\ell\not=j$},\\
\xb_j+\frac{c}{2}\xb_i\quad&\text{if $\ell=j$}.
\end{cases}
\end{equation}
This clearly induces an isomorphism
$\psi\colon \Z[\xa_1,\dots,\xa_n]\to \Z[\xb_1,\dots,\xb_n]$
between polynomial rings.
We claim that $\psi$ induces a graded isomorphism from $H^*(M(A);\Z)$ to $H^*(M(B);\Z)$.  Indeed, when $k\not=j$, it follows from \eqref{eq:4.0-2} and \eqref{eq:4.0-1} that
\begin{equation} \label{eq:4.0-3}
\psi(\ala_k)=\psi(\sum_{\ell=1}^nA_k^\ell\xa_\ell)=\sum_{\ell=1}^nA_k^\ell\xb_\ell+\frac{c}{2}A^j_k\xb_i=\alb_k
\end{equation}
and hence
\begin{equation} \label{eq:4.0-4}
\psi\big(\xa_k(\xa_k-\ala_k)\big)=\xb_k(\xb_k-\alb_k)\quad \text{when $k\not=j$.}
\end{equation}
When $k=j$, we have
\begin{equation}\label{eq:4.0-5}
\begin{split}
\psi\big(\xa_j(\xa_j-\ala_j)\big)&=\psi\big(\xa_j(\xa_j-c(\xa_i-\frac{1}{2}\ala_i))\big)\quad \hspace{26pt}\text{by \eqref{eq:4.0-0}}\\
&=(\xb_j+\frac{c}{2}\xb_i)(\xb_j-\frac{c}{2}\xb_i+\frac{c}{2}\alb_i)\quad\text{by \eqref{eq:4.0-2} and \eqref{eq:4.0-3}}\\
&=\xb_j(\xb_j+\frac{c}{2}\alb_i)-\frac{c^2}{4}\xb_i(\xb_i-\alb_i)\\
&=\xb_j(\xb_j-\alb_j)-\frac{c^2}{4}\xb_i(\xb_i-\alb_i)\quad\hspace{2pt}\text{by \eqref{eq:4.0-1-1}}.
\end{split}
\end{equation}
\eqref{eq:4.0-4} and \eqref{eq:4.0-5} together with \eqref{eq:1.2} show that $\psi$ induces a graded ring isomorphism from $H^*(M(A);\Z)$ to $H^*(M(B);\Z)$.

Since $\height(\ala_j)=i$ by \eqref{eq:4.0-0} and $\height(\alb_j)=\height(\alb_i)<i$ by \eqref{eq:4.0-1-1}, we have $\height(\alb_j)<\height(\ala_j)$.
Moreover, the isomorphism $\psi$ defined in \eqref{eq:4.0-2}
is represented as a unipotent upper triangular matrix with respect to the basis $\xa_1,\dots,\xa_n$ of $H^2(M(A);\Z)$ and $\xb_1,\dots,\xb_n$ of $H^2(M(B);\Z)$, so $\q_j=1$ for any $j$ and hence $\psi(p(M(A)))=p(M(B))$ by Lemma~\ref{lemm:3.1}.
\end{proof}

\begin{rema} \label{rema:4.1}
The graded ring isomorphism $\psi$ in the proof above is actually induced from a diffeomorphism, which follows from Theorem~\ref{theo:ishi} mentioned in Section~\ref{sect:6}.  One can also see it using the quotient construction of Bott manifolds.
\end{rema}

\section{Proof of Theorem~\ref{theo:0.1}} \label{sect:5}

This section is devoted to the proof of Theorem~\ref{theo:0.1} in the Introduction.

Let $\psi\colon H^*(M(A);\Z)\to H^*(M(B);\Z)$ be a graded ring isomorphism. By \eqref{eq:2.1-1} what we must prove is
\[
\psi\big(\prod_{j=1}^n(1+4(\ya_j)^2)\big)=\prod_{j=1}^n(1+4(\yb_j)^2)\in H^*(M(B);\Z).
\]
We may assume that none of $\ala_1,\dots,\ala_n$ and $\alb_1,\dots,\alb_n$ are of even exceptional type by Lemma~\ref{lemm:4.1}. Since $\psi(\ya_j)=\q_j\yb_{\sigma(j)}$ by Proposition~\ref{prop:3.1}, $\psi(1+4(\ya_j)^2)=1+4(\yb_{\sigma(j)})^2$ if $\q_j=\pm 1$.  Therefore, we shall treat the case where $\q_j=\pm\frac{1}{2}$ or $\pm 2$ by Lemma~\ref{lemm:3.1}.

Suppose $\q_j=\pm \frac{1}{2}$.  Then there is $i<j$ such that
\begin{equation} \label{eq:4.3}
\ala_{j}=c\ya_i \quad\text{and}\quad \alb_{\sigma(i)}=d\yb_{\sigma(j)}
\end{equation}\label{eq:4.4}
 with some odd integers $c$, $d$ and $\q_i=\pm 2$ by Lemma~\ref{lemm:3.2}. We shall show that
 \begin{equation} \label{eq:4.5}
 \psi\big((1+4(\ya_j)^2)(1+4(\ya_i)^2)\big)=(1+4(\yb_{\sigma(j)})^2)(1+4(\yb_{\sigma(i)})^2).
 \end{equation}
It follows from \eqref{eq:2.1-0} and \eqref{eq:4.3} that
\begin{equation} \label{eq:4.5-1}
4(\ya_j)^2=(\ala_j)^2=c^2(\ya_i)^2
\end{equation}
and sending the first and last elements in the identity above by $\psi$, we obtain
\begin{equation} \label{eq:4.6}
 (\yb_{\sigma(j)})^2=4c^2(\yb_{\sigma(i)})^2
 \end{equation}
 since $\q_j=\pm\frac{1}{2}$ and $q_i=\pm 2$.  Here $4(\yb_{\sigma(i)})^2=(\alb_{\sigma(i)})^2$ by \eqref{eq:2.1-0} and $(\alb_{\sigma(i)})^2=d^2(\yb_{\sigma(j)})^2$ by \eqref{eq:4.3}.  Therefore, \eqref{eq:4.6} turns into
\begin{equation} \label{eq:4.7}
(\yb_{\sigma(j)})^2=c^2d^2(\yb_{\sigma(j)})^2.
\end{equation}

When $(\yb_{\sigma(i)})^2\not=0$, \eqref{eq:4.7} implies $c^2d^2=1$ and hence $c^2=d^2=1$ because $c,d$ are integers.  It follows from \eqref{eq:4.5-1} and \eqref{eq:4.6} that
\[
4(\ya_j)^2=(\ya_i)^2\quad\text{and}\quad  (\yb_{\sigma(j)})^2=4(\yb_{\sigma(i)})^2.
\]
Therefore
\[
\begin{split}
\text{the left hand side of \eqref{eq:4.5} }&=\psi\big((1+(\ya_i)^2)(1+16(\ya_j)^2\big)\\
&=(1+\q_i^2(\yb_{\sigma(i)})^2)(1+16\q_j^2(\yb_{\sigma(j)})^2)\\
&=\text{the right hand side of \eqref{eq:4.5}}
\end{split}
\]
because $\q_j=\pm\frac{1}{2}$ and $\q_i=\pm 2$.  When $(\yb_{\sigma(i)})^2=0$, we have $(\yb_{\sigma(j)})^2=0$ by \eqref{eq:4.6} and $(\ya_j)^2=(\ya_i)^2=0$ because $\psi(\ya_j)=\yb_{\sigma(j)}$ and $\psi(\ya_i)=\yb_{\sigma(i)}$ and $\psi$ is an isomorphism.   Therefore, \eqref{eq:4.5} holds even when $(\yb_{\sigma(i)})^2=0$.

When $\q_j=\pm 2$, $\psi^{-1}(\yb_{\sigma(j)})=\pm\frac{1}{2}\ya_j$.  Therefore, the same argument as above applied for $\psi^{-1}$ shows that \eqref{eq:4.5} also holds when $\q_j=\pm 2$.  This completes the proof of the theorem.

\section{Proof of Theorem~\ref{theo:0.2}} \label{sect:6}

The purpose of this section is to prove Theorem~\ref{theo:0.2} in the Introduction.  We begin with the following lemma.

\begin{lemm} \label{lemm:6.1}
Let $A\in \MB_n$, $\sigma$ a permutation on $[n]$ and let $P$ be the permutation matrix of $\sigma^{-1}$, that is, the $(i,j)$ entry of $P$ is 1 if $i=\sigma(j)$ and 0 otherwise.  If $PAP^{-1}\in \MB_n$, then there is
 a graded ring isomorphism $\psi_{\sigma}\colon H^*(M(A);\Z)\to H^*(M(PAP^{-1});\Z)$ sending $\xa_{j}$ to $\xp_{\sigma(j)}$ for $j=1,2,\dots,n$ and
 it is induced from a diffeomorphism.
\end{lemm}

\begin{proof}
Remember the quotient construction of Bott manifolds explained in Section~\ref{sect:2}.
Let $\varphi_{\sigma}\colon (S^3)^n\to (S^3)^n$ be the coordinate change defined by
\[
\begin{split}
&\varphi_{\sigma}((z_1,w_1),(z_2,w_2),\dots,(z_n,w_n)\big)\\
=&\big((z_{\sigma(1)},w_{\sigma(1)}),(z_{\sigma(2)},w_{\sigma(2)}),\dots,(z_{\sigma(n)},w_{\sigma(n)})\big)
\end{split}
\]
where we consider the $(S^1)^n$-action associated to the matrix $PAP^{-1}$ on the source space and the one associated to $A$ on the target space, and let $\phi_\sigma$ be the group automorphism of $(S^1)^n$ defined by
\begin{equation*}
\phi_{\sigma}(g_1,g_2,\dots,g_n)=(g_{\sigma(1)},g_{\sigma(2)},\dots,g_{\sigma(n)}).
\end{equation*}
Then, $\varphi_{\sigma}$ is $\phi_\sigma$-equivariant, i.e.,
\begin{equation} \label{eq:6.0}
\varphi_\sigma(g\cdot (z,w))=\phi_\sigma(g)\cdot\varphi_\sigma((z,w)),
\end{equation}
where $g=(g_1,\dots,g_n)$ and $(z,w)=((z_1,w_1),(z_2,w_2),\dots,(z_n,w_n))$.
Indeed, the $j$-th component of the left hand side of \eqref{eq:6.0} is
\begin{equation} \label{eq:6.1}
\Big(\big(\prod_{k=1}^{\sigma(j)-1}g_k^{-(PAP^{-1})_{\sigma(j)}^k}\big)g_{\sigma(j)}z_{\sigma(j)},g_{\sigma(j)}w_{\sigma(j)}\Big)
\end{equation}
while that of the right hand side of \eqref{eq:6.0} is
\begin{equation} \label{eq:6.2}
\Big(\big(\prod_{i=1}^{j-1}g_{\sigma(i)}^{-A^i_j}\big)g_{\sigma(j)}z_{\sigma(j)},g_{\sigma(j)}w_{\sigma(j)}\Big).
\end{equation}
Here, since $P$ is the permutation matrix of $\sigma^{-1}$, $(PAP^{-1})^{\sigma(i)}_{\sigma(j)}=A^{i}_{j}$ and it is zero for $\sigma(i)\ge \sigma(j)$ or $i\ge j$ since both $PAP^{-1}$ and $A$ belong to $\MB_n$; so we have
\[
\prod_{k=1}^{\sigma(j)-1}g_{k}^{-(PAP^{-1})^k_{\sigma(j)}}=\prod_{k=1}^{n}g_{k}^{-(PAP^{-1})^k_{\sigma(j)}}=\prod_{i=1}^{n}g_{\sigma(i)}^{{-(PAP^{-1})}^{\sigma(i)}_{\sigma(j)}}=\prod_{i=1}^{n}g_{\sigma(i)}^{-A^i_j}=\prod_{i=1}^{j-1}g_{\sigma(i)}^{-A^i_j}.
\]
This shows that \eqref{eq:6.1} and \eqref{eq:6.2} agree, which means that $\varphi_{\sigma}$ is $\phi_{\sigma}$-equivariant and hence $\varphi_{\sigma}$ induces a diffeomorphism from $M(PAP^{-1})$ to $M(A)$.

It remains to prove that the cohomology ring isomorphism induced by $\varphi_{\sigma}$ maps $\xa_j$ to $\xp_{\sigma(j)}$.  Remember that  $\xa_{j}$ and $\xp_{\sigma(j)}$  are the first Chern classes of the complex line bundles $\gamma_{j}^A$ and $\gamma_{\sigma(j)}^{PAP^{-1}}$ mentioned in Section~\ref{sect:2}, where the matrices $A$ and $PAP^{-1}$ are specified to avoid confusion.  Therefore, it suffices to find a bundle isomorphism from $\gamma_{\sigma(j)}^{PAP^{-1}}$ to $\gamma_{j}^A$ which covers $\varphi_\sigma$. The map $f_\sigma$ from $(S^3)^n\times \C$ to itself defined by
\[
f_\sigma((z,w),u)=(\varphi_\sigma(z,w), u)
\]
satisfies
\[
f_\sigma(g\cdot(z,w),g_{\sigma(j)}^{-1}u)=(\phi_{\sigma}(g)\cdot(z,w),g^{-1}_{\sigma(j)}u).
\]
Since $g_{\sigma(j)}$ is the $j$-th component of $\phi_{\sigma}(g)$, $f_\sigma$ induces the desired bundle isomorphism from $\gamma_{\sigma(j)}^{PAP^{-1}}$ to $\gamma_{j}^A$.
\end{proof}

The following theorem due to H. Ishida plays a role in our argument.

\begin{theo}[\cite{ishi12}] \label{theo:ishi}
Let $A,B\in \MB_n$ and $\psi\colon H^*(M(A);\Z)\to H^*(M(B);\Z)$ be a graded ring isomorphism.  If $\psi$ restricted to the degree two cohomology groups 
is represented as an upper triangular matrix with respect to the basis $\xa_1,\dots,\xa_n$ of $H^2(M(A);\Z)$ and $\xb_1,\dots,\xb_n$ of $H^2(M(B);\Z)$ defined in Section~\ref{sect:2}, then $\psi$ is induced from a diffeomorphism.
\end{theo}

Let $\psi:H^*(M(A);\Z) \to H^*(M(B);\Z)$ be a graded ring isomorphism. By Proposition~\ref{prop:3.1}, there is a permutation $\sigma$ on $[n]$ such that $\psi(y_j^A)=\q_j y_{\sigma(j)}^B$ with some nonzero $\q_j \in \Q$ for any $j$.


\begin{prop} \label{prop:6.1}
If $q_j \in \{\pm 1\}$ for all $j$, then the isomorphism $\psi$ above is induced from a diffeomorphism.
\end{prop}

\begin{proof}
We may assume that none of $\ala_1,\dots,\ala_n$ are of even exceptional type by Lemma~\ref{lemm:4.1} and Remark~\ref{rema:4.1} (replace $A$ with the matrix given in Lemma~\ref{lemm:4.1} if necessary). Let $a_j^i$ be the $(i,j)$ entry of $(E-\frac 1 2 A)^{-1}$. Because none of $\ala_1,\dots,\ala_n$ are of even exceptional type and $\q_k \in \{\pm 1\}$ for all $k$, we have $a_j^i=0$ for $\sigma(i)> \sigma(j)$ by Lemma~\ref{lemm:3.1-1}.
This means that if $P$ is the permutation matrix of $\sigma^{-1}$, then $P(E-\frac{1}{2}A)^{-1}P^{-1}$ is a unipotent upper triangular matrix because the $(\sigma(i),\sigma(j))$ entry of  $P(E-\frac{1}{2}A)^{-1}P^{-1}$ is equal to the $(i,j)$ entry $a^i_j$ of $(E-\frac{1}{2}A)^{-1}$.  Then, since
$
P(E-\frac{1}{2}A)P^{-1}=\big(P(E-\frac{1}{2}A)^{-1}P^{-1}\big)^{-1}
$
is a unipotent upper triangular matrix, $PAP^{-1}$ is a strictly upper triangular matrix and hence $PAP^{-1}\in \MB_n$.

Let $\psi_{\sigma}\colon H^*(M(A);\Z)\to H^*(M(PAP^{-1});\Z)$ be the graded ring isomorphism in Lemma~\ref{lemm:6.1}.  Then the graded ring isomorphism arising from the composition
\[
\psi\circ \psi_{\sigma}^{-1}\colon H^*(M(PAP^{-1});\Z)\to H^*(M(A);\Z)\to H^*(M(B);\Z)
\]
satisfies $(\psi\circ\psi_{\sigma}^{-1})(\yp_j)=\q_{\sigma^{-1}(j)}\yb_j$ for any $j$ since $\psi_\sigma(\xa_i)=\xp_{\sigma(i)}$ for any $i$.  This means that  $\psi\circ \psi_{\sigma}^{-1}$ satisfies the assumption in Theorem~\ref{theo:ishi} and hence is induced from a diffeomorphism.  Moreover, $\psi_{\sigma}$ is induced from a diffeomorphism by Lemma~\ref{lemm:6.1}. This shows that $\psi$ is induced from a diffeomorphism.
\end{proof}


Remember that $M(O)=(\C P^1)^n$ for the zero matrix $O\in \MB_n$.  We say that a Bott manifold $M(A)$ for $A\in \MB_n$ is \emph{$\Z/2$-trivial} if $H^*(M(A);\Z/2)\cong H^*(M(O);\Z/2)$ as graded rings.

\begin{lemm} \label{lemm:6.0}
Let $A\in \MB_n$.  Then the following three statements are equivalent.
\begin{enumerate}
\item $M(A)$ is $\Z/2$-trivial.
\item $\ala_j\equiv 0 \pmod 2$ for $j=1,2,\dots,n$.
\item $A\equiv O\pmod 2$, i.e., every entry of $A$ is an even integer.
\end{enumerate}
\end{lemm}

\begin{proof}
The equivalence $(2)\Leftrightarrow(3)$ and the implication $(2)\Rightarrow (1)$ are obvious, so it suffices to prove the implication $(1)\Rightarrow (2)$.  Suppose that $M(A)$ is $\Z/2$-trivial.  Then the square of any element in $H^2(M(A);\Z/2)$ vanishes.  Therefore $\ala_j\xa_j(={\xa_j}^2)$ vanishes in $H^4(M(A);\Z/2)$ for any $j$.  On the other hand, the set $\{\xa_i\xa_j\mid 1\le i<j\le n\}$ is an additive basis of $H^4(M(A);\Z/2)$ because it is an additive basis of $H^4(M(A);\Z)$ by Lemma~\ref{lemm:1.1}.  Since $\ala_j$ is a linear combination of $\xa_i$ for $1\le i<j$, this implies that $\ala_j=0$ in $H^2(M(A);\Z/2)$ for any $j$, proving (2).
\end{proof}


Now we are in a position to prove our second main result stated in the Introduction.

\begin{proof}[Proof of Theorem~\ref{theo:0.2}.]
Let $M(A)$ and $M(B)$ be $\Z/2$-trivial Bott manifolds and let $\psi\colon H^*(M(A);\Z)\to H^*(M(B);\Z)$ be any graded ring isomorphism.
Since $M(A)$ and $M(B)$ are $\Z/2$-trivial, $\ala_j\equiv 0\pmod 2$ and $\alb_j\equiv 0\pmod 2$ for any $j$ by Lemma~\ref{lemm:6.0}.  This together with Lemma~\ref{lemm:3.1} shows that the assumption of Proposition~\ref{prop:6.1} is satisfied for the $\psi$, so the theorem follows from Proposition~\ref{prop:6.1}.
\end{proof}


\section{Concluding remarks} \label{sect:7}

We conclude this paper with some remarks on automorphisms of the cohomology ring of a Bott manifold $B_n$.  Proposition~\ref{prop:3.1} and Lemma~\ref{lemm:3.1} say that given a graded ring automorphism $\psi$ of $H^*(B_n;\Z)$, there are a permutation $\sigma$ on $[n]$ and $\q_j\in \{\pm\frac{1}{2}, \pm 1,\pm 2\}$ such that $\psi(y_j)=\q_jy_{\sigma(j)}$ for each $j=1,2,\dots,n$. Therefore, assigning $(\q_j/|\q_j|)_{j=1}^n$ together with $\sigma$ to $\psi$, we obtain a monomorphism
\begin{equation} \label{eq:7.1}
\Aut(H^*(B_n;\Z))\hookrightarrow \{\pm 1\}^n\rtimes \frak S_n.
\end{equation}
Here $\Aut(H^*(B_n;\Z))$ denotes the group of graded ring automorphisms of $H^*(B_n;\Z)$ and $\{\pm 1\}^n\rtimes \frak S_n$ is the signed permutation group on $[n]$, that is, the semidirect product of the $n$-fold product $\{\pm 1\}^n$ of the order two group $\{\pm 1\}$ and the permutation group $\frak S_n$ on $[n]$ where the action of $\frak S_n$ on $\{\pm 1\}^n$ is the natural permutation of factors of $\{\pm 1\}^n$. Lemma 4.2 in \cite{ishi12} implies that
any automorphism in the subgroup $\{\pm 1\}^n$ of $\{\pm 1\}^n\rtimes \frak S_n$ can be realized by diffeomorphisms of $B_n$ and our Lemma~\ref{lemm:6.1} implies that some elements of $\frak S_n$ can also be realized by diffeomorphisms of $B_n$.
However, any $\q_j$ is equal to $\pm 1$ for the cohomology automorphisms induced from those diffeomorphisms,
so if some $q_j$ is not $\pm 1$ then our results do not prove that the cohomology automorphism is induced from a diffeomorphism.
This is the reason why we needed to assume the $\Z/2$-triviality in Theorem~\ref{theo:0.2}.

As remarked in the Introduction, the Hirzebruch surface $\C P^2\sharp \overline{\C P^2}$ is not $\Z/2$-trivial (although it is $\Q$-trivial).  If we put $B_2=\C P^2\sharp \overline{\C P^2}$, then
\[
H^*(B_2;\Z)=\Z[x_1,x_2]/(x_1^2, x_2^2-x_1x_2),
\]
so $y_1=x_1$ and $y_2=x_2-\frac{1}{2}x_1$.  The map \eqref{eq:7.1} above is an isomorphism in this case.  In fact, $\Aut(H^*(B_2;\Z))$ is generated as a group by the following three automorphisms:
\begin{enumerate}
\item $(x_1,x_2)\to (-x_1,-x_1+x_2)$, so $(y_1,y_2)\to (-y_1,y_2)$,
\item $(x_1,x_2)\to (x_1,x_1-x_2)$, so $(y_1,y_2)\to (y_1,-y_2)$,
\item $(x_1,x_2)\to (-x_1+2x_2,x_2)$, so $(y_1,y_2)\to (2y_2,\frac{1}{2}y_1)$.
\end{enumerate}
The automorphisms (1) and (2) generate the subgroup $\{\pm 1\}^2$ while the automorphism (3) generates the subgroup $\frak S_2$.  As remarked above, the automorphisms (1) and (2) are induced from diffeomorphisms of $B_2$.  One can see that the automorphism (3) is also induced from a diffeomorphism of $B_2$, in fact, the diffeomorphism of Type 2 in the proof of \cite[Lemma 5.2]{ch-ma12} induces the automorphism (3) up to sign.  For the convenience of the reader, we shall review the construction of the diffeomorphism.
We take involutions
\[
\text{$s\colon [z_1,z_2,z_3]\to [\bar z_1,\bar z_2,\bar z_3]$ on $\C P^2$, \quad $t\colon [z_1,z_2,z_3]\to [z_1,z_2,-z_3]$ on $\overline{\C P^2}$}
\]
where $[z_1,z_2,z_3]$ denotes the homogeneous coordinate of $\C P^2$ and $\bar{z}$ denotes the conjugate of a complex number $z$. The fixed point set by $s$ is $\R P^2$ while that by $t$ consists of a point and $\C P^1$.  Choose a point from $\R P^2$ and $\C P^1$ respectively and take the equivariant connected sum of $\C P^2$ and $\overline{\C P^2}$ around the chosen fixed points.  Then the resulting involution on $B_2=\C P^2\sharp \overline{\C P^2}$ is the desired one, see the proof of \cite[Lemma 5.2]{ch-ma12} for more details.
As you see, this construction heavily depends on the description $\C P^2\sharp \overline{\C P^2}$ 
and we do not know how to find such a diffeomorphism for a not $\Z/2$-trivial Bott manifold $B_n$ although such a diffeomorphism is found when $B_n$ is $\Q$-trivial (\cite{ch-ma12}).  If we overcome this difficulty, then we could solve the strong cohomological rigidity conjecture for Bott manifolds completely.

\end{document}